\documentclass{amsart}

\usepackage{amssymb}

\newtheorem{thm}{Theorem}[section]
\newtheorem{prop}[thm]{Proposition}
\newtheorem{lem}[thm]{Lemma}
\newtheorem{cor}[thm]{Corollary}

\theoremstyle{definition}
\newtheorem{defi}[thm]{Definition}

\theoremstyle{remark}
\newtheorem{rem}[thm]{Remark}

\numberwithin{equation}{section}

\newcommand{\R}{\mathbb{R}^n}

\begin{document}

\title{$W^{\sigma,\epsilon}$-estimates for nonlocal elliptic equations}

\author{Hui Yu}
\address{Department of Mathematics, the University of Texas at Austin}
\email{hyu@math.utexas.edu}

\begin{abstract}
We prove a $W^{\sigma,\epsilon}$-estimate for a class of nonlocal fully nonlinear elliptic equations by following Fanghua Lin's original approach \cite{L} to the analogous problem for second order elliptic equations, by first proving a potential estimate, then combining this estimate with the ABP-type estimate by N. Guillen and R. Schwab to control the size of the superlevel sets of the $\sigma$-order derivatives of solutions.
\end{abstract}

 \maketitle

\tableofcontents

\section{Introduction}
Let $M^{+}_{2}$ and $M^{-}_{2}$ be the second order extremal Pucci operators \cite{CC}, then a classical result by Lin \cite{L}  states that for some universal $\epsilon$, one has an $\mathcal{L}^{\epsilon}$-estimate on the Hessian of a function satisfying two differential inequalities:   

\begin{thm}\text{[Lin's $W^{2,\epsilon}$-estimate]}

There exists universal constants $\epsilon>0$ and $C$ such that if $$M^{+}_{2}u\ge f\ge M^{-}_{2}u$$ in $B_1$, then\begin{equation}
\|u\|_{W^{2,\epsilon}(B_{1/2})}\le C(\|u\|_{\mathcal{L}^{\infty}(B_1)}+\|f\|_{\mathcal{L}^n(B_1)} ).
\end{equation}
\end{thm} 

Since the two inequalities impose very mild restrictions on $u$, this estimate is among the fundamental tools in the regularity theory of second order elliptic equations. For instance, it is the starting point of Caffarelli's $W^{2,p}$-estimate for solutions to fully nonlinear elliptic equations \cite{C} . Recently it was also used by Armstrong, Silvestre and Smart \cite{ASS} for their partial regularity result for fully nonlinear second order elliptic equations. 

The idea of the proof for Theorem 1.1, as presented in \cite{CC}, is a very clever use of the ABP-type estimate, which basically says that a function satisfying $M^{-}_{2}u\le f$ must touches its convex envelope in a contact set $\Gamma$ with large measure. Also note that on $\Gamma$ we have one-sided control on the Hessian of $u$ since it is touching a convex function that has nonnegative Hessians. The other inequality $M^+_2 u\ge f$ gives control from the other side, and together they imply that on a large set the Hessian of $u$ is small. And an induction argument gives the smallness of $|\{|D^2u|>t\}|$ for all $t>0$, which is enough for an $\mathcal{L}^{\epsilon}$-estimate of $|D^2 u|$. Note that in this argument we used a very delicate structure of the envelop, namely, its Hessian has a sign at every point.

If one wishes to extend this argument to nonlocal equations, one difficulty is the lack of a \textit{good} ABP-type estimate. The first nonlocal version of ABP-type estimate is given by Caffarelli and Silvestre \cite{CS1}, which says that if a function satisfies a differential inequality with a small right-hand side, then its \textit{enlarged} contact set with its convex envelop is big in measure. This is in itself a fundamental estimate for nonlocal equations, and  is the starting point of  regularity theory of fully nonlinear nonlocal elliptic equations. However, it only gives estimate on an enlarged version of $\Gamma$, on which one does not have a smallness of $D^2u$. Another disadvantage of this version of ABP-type estimate is that it only sees the $\mathcal{L}^{\infty}$-norm of the right-hand side. For instance, this does not tell the difference between $\chi_E$ and $\chi_F$, even when $|E|$ is much larger than $|F|$. As a result, it is not accurate enough to estimate the measure of superlevel sets of $D^2u$.

On the other hand, one does not expect a nice control on $D^2u$, since a nonlocal equation is a much \textit{softer} than second order equations. Instead, one expects to have estimate of $\sigma$-order. To this end, we have another replacement for ABP-type estimate, which is discovered by Guillen and Schwab in \cite{GS}.  There, instead of the convex envelop, they used a $\sigma$-order envelop, given as a solution to a fractional order obstacle problem with $u$ as the obstacle. By doing this they have estimate on the true contact set $\Gamma_{\sigma}$. Another advantage of this estimate is that they used the $\mathcal{L}^{n}$-norm of the right-hand side, which is suitable for estimating measure of superlevel sets. This is the main reason why we shall be using this version of ABP-type estimate. 

However, there are some disadvantages too. For one thing, their class of operators is in a sense more restrictive than the most natural class considered by \cite{CS1}. It remains open as for now whether a nice ABP-type estimate remains true in that generality. Nevertheless, the class considered by Guillen and Schwab is rich enough to recover second order theory in the limit when $\sigma\to 2$.  

The other disadvantage is more fundamental. Since the new envelop is given by an obstacle problem that is yet to be fully understood, it is not easy to pass estimates from this envelop to our $u$. As a result, Guillen and Schwab listed the $W^{\sigma,\epsilon}$-estimates as one of the open problems in their paper, which remains open for any kind of fully nonlinear equations of fractional order.

In this work, we present a proof of a $W^{\sigma,\epsilon}$-estimate for a class of nonlocal elliptic operators. To be precise, the main result is 
\begin{thm}
Suppose $u\in\mathcal{L}^{\infty}(\R)\cap C(B_1)$ satisfies in $B_1$ the following inequalities $$M^{-}_{\sigma}u\le f\le M^{+}_{\sigma}u,$$ then there exist universal constants $\epsilon>0$ and $C$ such that \begin{equation}
\|D^\sigma u\|_{\mathcal{L}^{\epsilon}(B_{1/2})}\le C(\|u\|_{\mathcal{L}^{\infty}(\R)}+\|f\|_{\mathcal{L}^{\infty}(B_1)}^{\frac{2-\sigma}{2}}\|f\|_{\mathcal{L}^{n}(B_1)}^{\frac{\sigma}{2}}). 
\end{equation} 
\end{thm} 


See Section 2 for the definition of the $\sigma$-order Hessian $D^{\sigma}$ as well as the $\sigma$-order extremal operators $M^{-}_{\sigma}$ and $M^{+}_{\sigma}$.

As corollaries we have the following, which are different forms of $W^{\sigma,\epsilon}$-estimate that might be more applicable to certain situations. They are suggested to the author by Dennis Kriventsov.

\begin{cor}
Suppose $u\in\mathcal{L}^{\infty}(\R)\cap C(B_1)$ satisfies in $B_1$ the following inequalities 
$$\begin{cases}M^{+}_{\sigma}u(x)&\ge -f^{-}(x) \\M^{-}_{\sigma}u(x)&\le f^{+}(x)\end{cases}$$ then for the same universal constant $\epsilon$ and another universal constant $C$ one has 
\begin{equation*}
\|D^\sigma u\|_{\mathcal{L}^{\epsilon}(B_{1/2})}\le C(\|u\|_{\mathcal{L}^{\infty}(\R)}+\|f\|_{\mathcal{L}^{\infty}(B_1)}^{\frac{2-\sigma}{2}}\|f\|_{\mathcal{L}^{n}(B_1)}^{\frac{\sigma}{2}}). 
\end{equation*} 
\end{cor}
and 

\begin{cor}
Suppose $u\in\mathcal{L}^{\infty}(\R)\cap C(B_1)$ satisfies in $B_1$ the following inequalities 
$$\begin{cases}M^{+}_{\sigma}u(x)&\ge -K \\M^{-}_{\sigma}u(x)&\le K\end{cases}$$ then for the same universal constant $\epsilon$ and another universal constant $C$ one has 
\begin{equation*}
\|D^\sigma u\|_{\mathcal{L}^{\epsilon}(B_{1/2})}\le C(\|u\|_{\mathcal{L}^{\infty}(\R)}+K). 
\end{equation*} 
\end{cor}


As mentioned before, we use the ABP-type estimate discovered by Guillen and Schwab. The main difficulty is then how to pass estimates on their fractional order envelop to the function $u$. We avoid this by following Lin's original strategy \cite{L} , instead of the one in \cite{CC}.  This strategy consists of two steps.  The first step is a potential estimate, where one shows that for $G,$ the Green's function to a linear operator, one has \begin{equation}
\int_E G(x,y)dy\ge C|E|^{m},
\end{equation} for any $x\in B_{1/2}$. This estimate for second order equations was discovered independently by Evans \cite{E}, and Fabes and Stroock \cite{FS}. 

The second step is to apply this, and the ABP estimate, to the set $E=\{|D^2u|>t\},$ which gives a bound on the distribution of $|D^2u|$. A very nice feature of this argument is that one avoids using any delicate structure of the envelop, and hence suits very well for our purpose.

This paper is organized as follows: In Section 2, we give some basic definitions and review some known results that will be needed in our work; In Section 3, we prove a nonlocal analogue of the potential estimate, following the strategy of Evans \cite{E} ; In Section 4, we finish the proof by completing the second step argument as in Lin's strategy. It should also be noted that we do not deal with existence issues in this work, and only focus on the estimates. Thus the result can either be viewed as an a priori estimate, or be made rigorous by an regularization and approximation procedure.

\section{Preliminaries}
We first define our $\sigma$-order replacement for the Hessian matrix:

\begin{defi}
For $u$ satisfying $$\int |\delta u(x,y)|\frac{1}{|y|^{n+\sigma}}dy<\infty,$$ $D^\sigma u(x)$ is the matrix with $(i,j)$-entry $$D^{\sigma}_{ij}u(x)= (2-\sigma)\int \delta u(x,y)\frac{\langle y,e_i\rangle \langle y,e_j\rangle}{|y|^{n+\sigma+2}}dy.$$
\end{defi} Here $\delta u(x,y)=u(x+y)+u(x-y)-2u(x)$, and $\{e_i\}$ is the standard basis for $\R$.

These operators have the following nice localization property:

\begin{prop}
Let $\eta$ be a smooth cut-off function that is $1$ in $B_{3/4}$ and vanishes outside $B_1$, then for $p\ge 1$ one has $$\|D^{\sigma}u\|_{\mathcal{L}^p(B_{1/2})}\le \|D^{\sigma}(\eta u)\|_{\mathcal{L}^p(B_{1/2})}+C\|u\|_{\mathcal{L}^{\infty}(\R)};$$ for $0<p<1$ one has similarly $$ \|D^{\sigma}u\|_{\mathcal{L}^p(B_{1/2})}\le C(p)\|D^{\sigma}(\eta u)\|_{\mathcal{L}^p(B_{1/2})}+C(p)\|u\|_{\mathcal{L}^{\infty}(\R)}.$$
\end{prop} 

\begin{proof}
Denote $u^{1}=\eta u$ and $u^{2}=(1-\eta)u$. Then $u^{2}=0$ inside $B_{3/4}$. Thus for $x\in B_{1/2}$ one has \begin{align*}|D^{\sigma}u^2|(x)&\le (2-\sigma)\int_{|y|>1/4}|\delta u^2(x,y)|\frac{1}{|y|^{n+\sigma}}dy \\ &\le C\|u\|_{\mathcal{L}^{\infty}(\R)}.\end{align*}
Thus for $p\ge 1$,\begin{align*}\|D^{\sigma}u\|_{\mathcal{L}^p(B_{1/2})}&\le \|D^{\sigma}u^1\|_{\mathcal{L}^p(B_{1/2})}+\|D^{\sigma}u^2\|_{\mathcal{L}^p(B_{1/2})}\\ &\le \|D^{\sigma}u^1\|_{\mathcal{L}^p(B_{1/2})}+|B_{1/2}|\|D^{\sigma}u^2\|_{\mathcal{L}^{\infty}(B_{1/2})}\\ &\le  \|D^{\sigma}u^1\|_{\mathcal{L}^p(B_{1/2})}+C|B_{1/2}|\|u\|_{\mathcal{L}^{\infty}(\R)}.\end{align*}

For $0<p<1$, one uses instead \begin{equation*}\|D^{\sigma}u\|_{\mathcal{L}^p(B_{1/2})}\le C(p)\|D^{\sigma}u^1\|_{\mathcal{L}^p(B_{1/2})}+C(p)\|D^{\sigma}u^2\|_{\mathcal{L}^p(B_{1/2})}.\end{equation*} 
\end{proof}

The reader could see more properties of this operator in \cite{Y}. 

Now we define the extremal operators we use. 
\begin{defi}
Let $\mathcal{L}$ be the collection of kernels of the form $$K(y)=(2-\sigma)\frac{\langle A y, y\rangle}{|y|^{n+\sigma+2}},$$where $0<\lambda\le A\le\Lambda<\infty$. 

Then the extremal operators are defined by 
$$M^{-}_{\sigma}u(x)=\inf_{K\in\mathcal{L}}\int\delta u(x,y)K(y)dy,$$ and  $$M^{+}_{\sigma}u(x)=\sup_{K\in\mathcal{L}}\int\delta u(x,y)K(y)dy.$$

Also, for $0<\lambda\le A(x)\le\Lambda<\infty$ we denote $L_A$ the operator $$L_Au(x)=\int\delta u(x,y)\frac{(2-\sigma)\langle A(x)y,y\rangle}{|y|^{n+\sigma+2}}dy.$$
\end{defi} 

Note that this class of kernels is essentially the class considered by Guillen and Schwab in their ABP-type estimate, although they do allow some degeneracy by only assuming $\lambda\le\text{Trace}(A)$. Also note that a class of operators elliptic with respect to $\mathcal{L}$ has been recently shown to admit smooth solution \cite{Y} . This hints that this class is nice in the sense that results from second order theory pass relatively directly to this class, while it is still rich enough to recover the second order theory in the limit as $\sigma\to 2$.

These extremal operators also localize well:

\begin{prop}
Let $\eta$ be as in the previous proposition. If $$M^{-}_{\sigma}u(x)\le f(x)$$ in $B_1$, then in $B_{1/2}$ one has $$M^{-}_{\sigma}(\eta u)(x)\le f(x)+ C\|u\|_{\mathcal{L}^{\infty}(\R)}.$$ 
\end{prop} 

\begin{proof}
Let $u^1$ and $u^2$ be as in the last proof. Then for $x\in B_{1/2}$ \begin{align*}
M^{-}_{\sigma}u^2(x)&=\inf_{K\in\mathcal{L}}\int\delta u^2(x,y)K(y)dy\\&\ge -(2-\sigma)\Lambda\int_{|y|>1/4}|\delta u^2(x,y)|\frac{1}{|y|^{n+\sigma}}dy\\&\ge -C\|u\|_{\mathcal{L}^{\infty}(\R)}. 
\end{align*} 
Thus \begin{align*}
M^{-}_{\sigma}u^1(x)&\le M^{-}_{\sigma}u(x)-M^{-}_{\sigma}u^2(x)\\& \le f(x)+C\|u\|_{\mathcal{L}^{\infty}(\R)}. 
\end{align*} 
\end{proof}
We will also need the following version of the Carledr\'on-Zygmund decomposition as an inductive tool. Throughout this paper we denote by $Q(x;r)$ the cube centred at $x\in\R$, with sides parallel to coordinate axis and of length $2r$.
\begin{prop}
Suppose $E\subset Q(0;1)$ satisfies $|E|<\alpha|Q(0;1)|$ for some $\alpha\in(0,1)$. Then there are cubes $\{\widetilde{Q_j}\}$ with mutually disjoint interior covering $E$ almost everywhere, and $$|\cup\widetilde{Q_j}|>1/\alpha|E|.$$ Moreover, each $\widetilde{Q_j}$ contains at least a dyadic subcube $Q_j$ such that $$|E\cap Q_j|\ge\alpha|Q_j|.$$
\end{prop} 

\begin{proof}
Begin with $Q(0;1)$, we divide a cube dyadically if $|Q\cap E|<\alpha|Q|$. And we keep a cube if this inequality fails. In particular our hypothesis says $Q(0;1)$ is divided.

Let $\{Q_j\}$ denote the collection of cubes that we keep. Then $\cup Q_j$ covers $E$ up to a null set, since any point outside $\cup Q_j$ is contained in a sequence of nested cubes with $|Q\cap E|/|Q|<\alpha<1$. 

Now let $\widetilde{Q_j}$ be the dyadic predecessor of $Q_j$. If several $Q_j$'s share the same predecessor then we just pick one. Then obviously $\widetilde{Q_j}$'s have mutually disjoint interior. Also they cover $E$ up to a null set.

Moreover since each $\widetilde{Q_j}$ is further divided, one has $|\widetilde{Q_j}\cap E|<\alpha|\widetilde{Q_j}|,$ then $$|\cup\widetilde{Q_j}|=\Sigma |\widetilde{Q_j}|>1/\alpha|\Sigma\widetilde{Q_j}\cap E|\ge 1/\alpha|E|.$$
\end{proof}

Finally we recall the nonlocal ABP-type estimate we shall be using in this work. It is a deep result by Guillen and Schwab \cite{GS}. We will need a scaled version of their original theorem. Note that their theorem is more general than the following version in the sense that they allow more degenerate kernels. Also their estimate is more accurate than the following because they only need information of $f$ on a $\sigma$-order contact set. However, we shall not need that in this work.

\begin{thm}
Assume $u\in\mathcal{L}^{\infty}(\R)\cap LSC(\R)$ satisfies 
$$\begin{cases}M^{-}_{\sigma}u\le f &\text{in $B_R$}\\u\ge 0 & \text{in $B_R^c$},\end{cases}$$then there is a constant $C(n)$ such that \begin{equation*}
-\inf_{B_R}u\le \frac{C(n)}{\lambda}\frac{1}{R^{\sigma/2}}\|f\|_{\mathcal{L}^{\infty}(B_R)}^{\frac{2-\sigma}{2}}\|f\|_{\mathcal{L}^n(B_R)}^{\frac{\sigma}{2}}.
\end{equation*} 
\end{thm} 

\begin{rem}
It is pointed out by T. Jin to the author that N. Guillen has an improved version of this result, where the right-hand side depends only on $\|f\|_{2n/\sigma}$. As a result, we have corresponding improvement of all results in this paper. 
\end{rem} 

This clearly implies the following corollary concerning solutions with non-vanishing boundary data:
\begin{cor}
Assume $u\in\mathcal{L}^{\infty}(\R)\cap LSC(\R)$ satisfies 
$$\begin{cases}M^{-}_{\sigma}u\le f &\text{in $B_R$}\\u\ge -B & \text{in $B_R^c$},\end{cases}$$
then there is a constant $C(n)$ such that \begin{equation*}
-\inf_{B_R}u\le \frac{C(n)}{\lambda}\frac{1}{R^{\sigma/2}}\|f\|_{\mathcal{L}^{\infty}(B_R)}^{\frac{2-\sigma}{2}}\|f\|_{\mathcal{L}^n(B_R)}^{\frac{\sigma}{2}}+B.
\end{equation*} 
\end{cor}

\begin{proof}Apply the theorem to $u+B$.\end{proof}

\section{The potential estimate}
The goal of this section is to prove the following theorem: \begin{thm} There are universal constants $\delta>0$ and $C$, such that 
for $u$ solving \begin{equation}
\begin{cases}L_A u(x)=\int \delta u(x,y)\frac{(2-\sigma)\langle A(x)y,y\rangle}{|y|^{n+\sigma+2}}dy=-\chi_E &\text{in $B_1$}\\ u=g &\text{outside $B_1$}\end{cases}
\end{equation} for some $E\subset B_{1/2}$ and $\lambda\le A(\cdot)\le\Lambda$ , one has $$\inf_{B_{1/2}}u\ge C|E|^{\delta}-\|g\|_{\mathcal{L}^{\infty}(\R)}.$$ 
\end{thm} For the theory of this estimate for second order equations, see \cite{E}. 

As remarked before, we do not deal with the existence issues, and hence our estimates are a priori in nature. They can also be made rigorous for viscosity solutions by a regularization and approximation procedure. 

Moreover, since we are dealing with a linear operator, superposition implies that it suffices to deal with solutions with $0$ boundary data, and to prove $$\inf_{B_{1/2}}u\ge C|E|^{\delta}$$ for such solutions.

We begin with the following: \begin{lem}
There are universal constants $\gamma,\beta\in (0,1)$ such that if $Q, E\subset B_{1/8}$ satisfies $$|Q\cap E|\ge\beta|Q|,$$ then the solutions to 
$$\begin{cases}L_A u=-\chi_{E\cap Q} &\text{in $B_1$}\\u=0&\text{in $B_1^c$},\end{cases}$$ and to 
$$\begin{cases}L_A v=-\chi_{3Q} &\text{in $B_1$}\\v=0&\text{in $B_1^c$}\end{cases}$$ satisfy the following estimate in $B_1$ $$u\ge\gamma v.$$ 
\end{lem} 

\begin{proof}
Suppose $Q=Q(x_0;l)$. Recall that this is a cube centred at $x_0$, with sides parallel to coordinate axis and with length $2l$. We first rescale the problem by defining $$\tilde u(x)=l^{-\sigma}u(x_0+l x)$$ and $$\tilde v(x)=l^{-\sigma}v(x_0+l x).$$ They solve the following equations, respectively,
$$\begin{cases}L_{\tilde{A}} \tilde{u}=-\chi_{\frac{E\cap Q-x_0}{l}} &\text{in $\frac{B_1-x_0}{l}$}\\ \tilde{u}=0&\text{outside},\end{cases}$$ and $$\begin{cases}L_{\tilde{A}} \tilde{v}=-\chi_{\frac{3Q-x_0}{l}} &\text{in $\frac{B_1-x_0}{l}$}\\ \tilde{v}=0&\text{outside}.\end{cases}$$Here $\tilde{A}(x)=A(x_0+lx)$ satisfies the same ellipticity condition. Also note that the original cube $Q(x_0;l)$ is rescaled to $Q(0;1)$.

Now let $w$ be the solution to 
$$\begin{cases}L_{\tilde{A}} w=-\chi_{Q(0;1)} &\text{in $\frac{B_1-x_0}{l}$}\\ w=0&\text{in outside}.\end{cases}$$
We compare $w$ to a barrier $\Phi\in C(\R)$ such that $$\begin{cases} \Phi(x)\ge C_{\Phi}>0 &\text{in $Q(0;6)$ }\\ \Phi=0 &\text{outside $B_{8\sqrt{n}}$}\\ M^{-}_{\sigma}\Phi\ge -\Psi&\text{in $\R$},\end{cases}$$ where $0\le\Psi\le 1$ and $\Psi=0$ outside $B_1$. 

For the existence of such a barrier, see \cite{CS1}.  

Since $l<\frac{1}{8\sqrt{n}}$, one has $B_{8\sqrt{n}}\subset \frac{B_1-x_0}{l}$. In particular $w\ge \Phi$ outside $B_{8\sqrt{n}}$. Also since $B_1\subset Q(0;1)$, inside $B_{8\sqrt{n}}$ one has $$L_{\tilde{A}}w\le M^{-}_{\sigma}\Phi\le L_{\tilde{A}}\Phi.$$ Thus comparison principle for linear equations gives \begin{equation*}
w\ge\Phi
\end{equation*} in $B_{8\sqrt{n}}$. In particular in $Q(0;6)$ one has \begin{equation}
w\ge C_{\Phi}.
\end{equation} 

Now $\tilde{u}-w$ satisfies 
$$\begin{cases}L_{\tilde{A}} (\tilde{u}-w)=\chi_{Q(0;1)\backslash (E-x_0/l)} &\text{in $\frac{B_1-x_0}{l}$}\\ \tilde{u}-w=0&\text{outside}.\end{cases}$$ In particular the equation implies 
$$M^{-}_{\sigma}(\tilde{u}-w)\le \chi_{Q(0;1)\backslash (E-x_0/l)},$$ hence Theorem 2.6 implies

\begin{align*}-\inf_{\frac{B_1-x_0}{l}} (\tilde{u}-w)&\le \frac{C(n)}{\lambda} l^{\sigma/2} \|\chi_{Q(0;1)\backslash (E-x_0/l)}\|_{\mathcal{L}^n(\frac{B_1-x_0}{l})}^{\sigma/2}\\&\le \frac{C(n)}{\lambda}l^{\sigma/2}|Q(0;1)\backslash (E-x_0/l)|^{\sigma/2n}\\&\le  \frac{C(n)}{\lambda}l^{\sigma/2}(1-\beta)^{\sigma/2n}\\&\le\frac{C(n)}{\lambda}(1-\beta)^{\sigma/2n}.
\end{align*}For the last inequality we used $l<1$.

In particular this implies that in $Q(0;6)$ one has \begin{equation}
\tilde{u}\ge w-\frac{C(n)}{\lambda}(1-\beta)^{\sigma/2n}\ge \frac{1}{2}C_{\Phi}
\end{equation} once we choose $\beta$ universally close to $1$.

Now we apply Theorem 2.6 to $-\tilde{v}$, which satisfies $$M^{-}_{\sigma}(-\tilde{v})\le L_{\tilde{A}}(-\tilde{v})=\chi_{\frac{3Q-x_0}{l}}$$ in $\frac{B_1-x_0}{l}$, to get \begin{equation}
\sup_{\frac{B_1-x_0}{l}}\tilde{v}\le\frac{C(n)}{\lambda}l^{\sigma/2}|\frac{3Q-x_0}{l}|^{\sigma/2n}\le\frac{C(n)}{\lambda}l^{\sigma/2}|Q(0;3)|^{\sigma/2n}\le \frac{C(n)}{\lambda}.
\end{equation} 

Combining the previous two estimates one has 
\begin{equation*}
\tilde{u}\ge C\tilde{v}
\end{equation*} in $Q(0;6)$. And hence $\tilde{u}\ge C\tilde{v}$ outside  $(\frac{B_1-x_0}{l}\backslash Q(0;6))$.

On the other hand, since $Q(0;6)\supset Q(0;3)=\frac{3Q-x_0}{l}$, we have $$L_A(\tilde{u}-C\tilde{v})=0$$ in $\frac{B_1-x_0}{l}\backslash Q(0;6)$. Consequently the comparison principle for linear equation gives us \begin{equation}
\tilde{u}\ge C\tilde{v}
\end{equation} in $\frac{B_1-x_0}{l}\backslash Q(0;6)$, and hence in $\frac{B_1-x_0}{l}$ since we already have the estimate in $Q(0;6)$.

Now rescale back to $u$ and $v$ to get the desired estimate.
\end{proof}

We now combine the previous lemma and the Carledr\'on-Zygmund decomposition to prove the potential estimate. This is the same strategy Evans used in \cite{E}. 

\begin{proof}
The proof is by an induction on the size of $E$. Also we prove the theorem when $E\subset Q(0;1/2)$ instead of $B_{1/2}$. This has no effect on the estimate after a covering argument.

If $|E|\ge\beta |Q(0;1/2)|$, then a similar argument as in the proof for the lemma (when we show $w\ge C$ in $Q(0;6)$) gives $u\ge C$.

Suppose we have proved the following for all $k\le k_0$:

If $|E|\ge\beta^{k}|Q(0;1/2)|$ then $u\ge C\gamma^{k}$. (*)

We proceed to prove the statement when $|E|\ge\beta^{k_0+1}|Q(0;1/2)|$. The Carledr\'on-Zygmund decomposition gives a covering of $E$ by essentially disjoint cubes $\{\widetilde{Q_j}\}$, each containing a dyadic subcube $Q_j$ with $|E\cap Q_j|\ge\beta|Q_j|$, and also $|\cup\widetilde{Q_j}|\ge\frac{1}{\beta}|E|$.

In particular, $|\cup \widetilde{Q_j}|\ge\beta^{k_0}|Q(0;1/2)|$ and thus we apply the induction hypothesis to conclude \begin{equation}
\inf_{Q(0;1/2)} v\ge C\gamma^{k_0},
\end{equation} where $v$ is the solution to 
$$\begin{cases}L_A v=-\chi_{\cup \widetilde{Q_j}} &\text{in $B_1$}\\v=0&\text{outside}.\end{cases}$$
 Also by the previous lemma one has $$u_j\ge\gamma v_j,$$ where $u_j$ and $v_j$ solve 
 $$\begin{cases}L_A u_j=-\chi_{Q_j\cap E} &\text{in $B_1$}\\u_j=0&\text{outside},\end{cases}$$
and $$\begin{cases}L_A v_j=-\chi_{3Q_j} &\text{in $B_1$}\\v_j=0&\text{outside}.\end{cases}$$

Now note that $3Q_j\supset\widetilde{Q_j}$, hence $v_j\ge\tilde{v}_j$, where $\tilde{v}_j$ solves 
 $$\begin{cases}L_A \tilde{v_j}=-\chi_{\widetilde{Q_j}} &\text{in $B_1$}\\ \tilde{v_j}=0&\text{outside}.\end{cases}$$
 
 To conclude, \begin{align*}
\inf_{Q(0;1/2)}u&=\inf_{Q(0;1/2)}\Sigma u_j\\&\ge\inf_{Q(0;1/2)}\Sigma\gamma v_j\\&\ge\inf_{Q(0;1/2)}\Sigma\gamma \tilde{v_j}\\&=\gamma\inf_{Q(0;1/2)}\Sigma\tilde{v_j}\\&=\gamma\inf_{Q(0;1/2)}v.
\end{align*} Combining this with (3.6) one completes the inductive step, and hence (*) is true for any $k$.

This implies $$\inf_{Q(0;1/2)}u\ge C|E|^{\log_\beta\gamma}.$$
\end{proof}

\section{$W^{\sigma,\epsilon}$-estimate}

In this section be begin our proof for the $W^{\sigma,\epsilon}$-estimate. As in the classical strategy, the key point is to control the size of $\{|D^{\sigma}u|>t\}$.  This is done by using two competing estimates: while an ABP-type estimate bounds the solution from above with the right-hand side, the previous potential estimate gives a bound from below. These two must balance for a function satisfying two differential inequalities.

We first prove several preparatory lemmas concerning solutions to linear equations. Our estimates will be independent of any regularity of the kernels, and hence they imply an estimate for solutions to the differential inequalities.

The first lemma is a technical device to realize the $\sigma$-order Hessian as the right-hand side of an equation.  This is essentially how we avoid using any delicate structure of the envelop and still get an estimate on $\{|D^{\sigma}u|>t\}$.

\begin{lem}
Suppose $u$ solves $$\begin{cases}L_Au=-f &\text{in $B_1$}\\u=g &\text{outside}\end{cases}$$ for some $\lambda\le A(\cdot)\le\Lambda$, then there is $\tilde{A}(\cdot)$ with $\frac{\lambda}{2}\le\tilde{A}(\cdot)\le 2\Lambda $ such that $$L_{\tilde{A}}u(x)=-f(x)-\nu |D^{\sigma}u|(x)$$ where $\nu\ge C(\lambda,\Lambda)$.
\end{lem} 

\begin{proof}
For each $x$ one finds $\frac{\lambda}{2}\le\tilde{A}(x)\le 2\Lambda$ such that 
$$\tilde{A}_{ij}(x)D^{\sigma}_{ij}u(x)=2\Lambda\Sigma_{e<0}e+\frac{\lambda}{2}\Sigma_{e>0}e,$$where $\{e\}$ are the eigenvalues of $D^{\sigma}u(x)$.

Then $$L_{\tilde{A}}u(x)=\int \delta u(x,y)\frac{(2-\sigma)\langle\tilde{A}(x)y,y\rangle}{|y|^{n+\sigma+2}}dy=2\Lambda\Sigma_{e<0}e+\frac{\lambda}{2}\Sigma_{e>0}e.$$

By ellipticity of $L_A$, one has 
$$L_A u(x)\ge \lambda\Sigma_{e>0}e+\Lambda\Sigma_{e<0}e.$$ Combining these two gives 
$$L_{\tilde{A}}u(x)-L_A u(x)\le \Lambda\Sigma_{e<0}e-\frac{\lambda}{2}\Sigma_{e>0}e\le -\min\{\Lambda,\frac{\lambda}{2}\}|D^{\sigma}u|(x).$$ 
\end{proof}

The second lemma gives a bound on $|\{|D^\sigma u|>t\}|$ for a solution with negative right-hand side:
\begin{lem}
Suppose $u$ solves $$\begin{cases}L_Au=-f\le 0 &\text{in $B_1$}\\u=g &\text{outside}\end{cases},$$ then \begin{equation}
|\{|D^{\sigma}u|>t\}\cap B_{1/2}|^{\delta}\le \frac{C}{t}(\|f\|_{\mathcal{L}^{\infty}(B_1)}^{\frac{2-\sigma}{2}}\|f\|_{\mathcal{L}^n(B_1)}^{\frac{\sigma}{2}}+\|g\|_{\mathcal{L}^{\infty}(\R)}),
\end{equation} where $\delta$ is as in the potential estimate and $C$ is a universal constant.
\end{lem} 

\begin{proof}
Let $\tilde{A}(\cdot)$ and $\nu$ be as in the previous lemma. Let $v$ be the solution to 
$$\begin{cases}L_{\tilde{A}}v=-\nu|D^{\sigma}u| &\text{in $B_1$}\\v=g &\text{outside}\end{cases}.$$ Then comparison principle for linear equations gives $$u\ge v.$$

For $t>0$, let $v_t$ be the solution to 
$$\begin{cases}L_{\tilde{A}}v_t=-\inf\nu\cdot t\chi_{\{|D^{\sigma}u|>t\}} &\text{in $B_1$}\\v_t=g &\text{outside}\end{cases}.$$ Then $t\chi_{\{|D^{\sigma}u|>t\}}\le |D^{\sigma}u|$ implies $v_t\le v$. Thus \begin{equation}
u\ge v_t.
\end{equation} 

Now let $w$ be the solution to $$\begin{cases}L_{\tilde{A}}w=-\chi_{\{|D^{\sigma}u|>t\}} &\text{in $B_1$}\\w=\frac{g}{\inf\nu\cdot t} &\text{outside}\end{cases}.$$

Then the potential estimate on $w$ gives

\begin{align*}
\inf_{B_{1/2}}v_t &={\inf\nu\cdot t}\inf_{B_{1/2}}w\\& \ge \inf\nu\cdot t(C|\{|D^{\sigma}u|>t\}|^{\delta}-\|\frac{g}{\inf\nu\cdot t}\|_{\mathcal{L}^{\infty}(\R)})\\&\ge Ct|\{|D^{\sigma}u|>t\}|^{\delta}-C\|g\|_{\mathcal{L}^{\infty}(\R)}. 
\end{align*} For the last inequality we used $\nu\ge C(\lambda,\Lambda)$.

Combining this with (4.2), and the ABP-type estimate for $u$ to obtain
\begin{align*}
\frac{C(n)}{\lambda}\|f\|_{\mathcal{L}^{\infty}(B_1)}^{\frac{2-\sigma}{2}}\|f\|_{\mathcal{L}^n(B_1)}^{\frac{\sigma}{2}}+\|g\|_{\mathcal{L}^{\infty}(\R)}&\ge \sup_{B_1} u \\&\ge \inf_{B_{1/2}}v_t\\&\ge Ct|\{|D^{\sigma}u|>t\}|^{\delta}-C\|g\|_{\mathcal{L}^{\infty}(\R)}.
\end{align*} This gives the desired estiamte.
\end{proof}

As remarked at the beginning of this section, since our estimate for linear equations depends only on universal properties, it is as good as one for the two differential inequalities. Hence we can finish the proof for the main result:


\begin{proof}We again let $g$ denote the boundary datum.

The two inequalities give $A^{+}(\cdot)$ and $A^{-}(\cdot)$ between $\lambda$ and $\Lambda$ such that 
$$\int \delta u(x,y)\frac{(2-\sigma)\langle A^{+}(x)y,y\rangle}{|y|^{n+\sigma+2}}dy\ge f(x)\ge \int \delta u(x,y)\frac{(2-\sigma)\langle A^{-}(x)y,y\rangle}{|y|^{n+\sigma+2}}dy.$$ In particular one finds $t:B_1\to [0,1]$ such that $$\int \delta u(x,y)\frac{(2-\sigma)\langle (t(x)A^{+}(x)+(1-t(x))A^{-}(x))y,y\rangle}{|y|^{n+\sigma+2}}dy=f(x).$$ Take $A(x)=t(x)A^{+}(x)+(1-t(x))A^{-}(x)$, then $\lambda\le A(\cdot)\le\Lambda$ and $u$ solves 
$$\begin{cases}L_A u=f &\text{in $B_1$}\\u=g &\text{outside}.\end{cases}$$ By linearity, $u$ is the sum of two functions, solving the equation with data $(f^{+},0)$ and $(-f^{-},g)$ respectively. By applying the previous lemma to each piece, we have $$|\{|D^{\sigma}u|>t\}\cap B_{1/2}|^{\delta}\le \frac{C}{t}(\|f\|_{\mathcal{L}^{\infty}(B_1)}^{\frac{2-\sigma}{2}}\|f\|_{\mathcal{L}^n(B_1)}^{\frac{\sigma}{2}}+\|g\|_{\mathcal{L}^{\infty}(\R)}).$$

This implies \begin{align*}\int_{B_{1/2}}|D^{\sigma}u|^{\epsilon}dx &\le C+\int_{1}^{\infty} t^{\epsilon-1}|\{|D^\sigma u|>t\}\cap B_{1/2}|dt\\ &\le C+\int_{1}^{\infty} t^{\epsilon-1}t^{-1/\delta}dt \cdot(\|f\|_{\mathcal{L}^{\infty}(B_1)}^{\frac{2-\sigma}{2}}\|f\|_{\mathcal{L}^n(B_1)}^{\frac{\sigma}{2}}+\|g\|_{\mathcal{L}^{\infty}(\R)})^{1/\delta}.
\end{align*}

The integral converges once we choose $\epsilon<1/\delta$, and this implies $$\|D^{\sigma}u\|_{\mathcal{L}^{\epsilon}(B_{1/2})}\le C(1+(\|f\|_{\mathcal{L}^{\infty}(B_1)}^{\frac{2-\sigma}{2}}\|f\|_{\mathcal{L}^n(B_1)}^{\frac{\sigma}{2}})^{1/\delta}+\|g\|_{\mathcal{L}^{\infty}(\R)})^{1/\epsilon}.$$ Then note that such an estimate must be scalable in $u$, $f$ and $g$, hence one has the desired estimate.
\end{proof}

We now prove Corollary 1.3 by reducing it to our main theorem by taking certain convex combinations of extremal operators. As mentioned in Introduction, this corollary is suggested to the author by Dennis Kriventsov.

\begin{proof}
Since $M^{+}_{\sigma}u(x)\ge -f^{-}(x)$ in $B_1$, one finds $A^{+}_1:B_1\to [\lambda,\Lambda]$ such that for all $x\in B_1$ $$L_{A^{+}_1}u(x)\ge -f^{-}(x).$$ 

Define $\Omega_1=\{x\in B_1: M^{-}_{\sigma}u(x)\le -2f^{-}(x)\}$, then there is $A^{+}_2:\Omega_1\to [\lambda.\Lambda]$ such that for all $x\in \Omega_1$  $$L_{A^+_2}u(x)\le -2f^{-}(x).$$
Consequently we can find $t_1^{+}:\Omega_1\to [0,1]$ such that for all $x\in\Omega_1$ $$t^{+}_1(x)L_{A^+_1}u(x)+(1-t^+_1(x))L_{A^+_2}u(x)=-\frac{3}{2}f^{-}(x).$$

Define $t^+:B_1\to [0,1]$ by $$\begin{cases} t^{+}=1 &\text{outside $\Omega_1$}\\t^{+}=t^{+}_1  &\text{inside $\Omega_1$},\end{cases}$$ and $A^{+}=t^{+}A_1^++(1-t^+)A_2^+$. Then $$L_{A^+}u(x)=\begin{cases}L_{A^+_1}u(x)\ge -f^{-}(x) &\text{outside $\Omega_1$}\\-\frac{3}{2}f^{-}(x) &\text{inside $\Omega_1$}.\end{cases}$$

Similarly one finds $A^-: B_1\to [\lambda,\Lambda]$ such that $$L_{A^-}u(x)=\begin{cases}\le f^{+}(x) &\text{outside $\Omega_2$}\\ \frac{3}{2}f^{+}(x) &\text{inside $\Omega_2$},\end{cases}$$ where $\Omega_2=\{x\in B_1: L_{A^+}u(x)\ge 2f^{+}(x)\}.$ In particular we can assume $$\Omega_1\cap\Omega_2=\emptyset.$$As a result we can define $t:B_1\to [0,1]$ such that $t=1$ in $\Omega_1$ and $t=0$ in $\Omega_2$. By taking $A=tA^++(1-t)A^-$ one has inside $\Omega_1$
\begin{align*}L_Au(x)&=t(x)L_{A^+}u(x)+(1-t(x))L_{A^-}u(x)\\&=L_{A^+}u(x)\\&\ge -f^{-}(x),\end{align*} and outside $\Omega_1$ 
\begin{align*}L_Au(x)&=t(x)L_{A^+}u(x)+(1-t(x))L_{A^-}u(x)\\&\ge -f^{-}(x)t(x)+(-2f^{-}(x))(1-t(x))\\&\ge -2f^{-}(x).\end{align*} To conclude $L_Au(x)\ge -2f^{-}(x)$ in $B_1$. Similarly $L_Au(x)\le 2f^{+}(x)$ in $B_1$. As a result $u$ solves in $B_1$ $$L_Au(x)=g(x)$$ where $-2f^{-}\le g\le 2f^+.$ 

Applying the main result to this equation clearly gives the desired estimate.
\end{proof}

Similar proof applies to Corollary 1.4.

\section*{Acknowledgement} The author would like to thank his PhD advisor, Luis Caffarelli, for many valuable conversations regarding this project. He is also grateful to his colleagues and friends, especially Dennis Kriventsov, Xavier Ros-Oton, Tianlin Jin and Luis Duque for all the discussions and encouragement. Kriventsov, Ros-Oton, and Jin also read a previous version of this work and gave many insightful comments. Corollary 1.3 is  suggested to the author by Kriventsov. Jin pointed to the author an improved version of Guillen-Schwab's estimate, see Remark 2.7.


\end{document}